\documentclass[12pt]{article}
\usepackage{multirow}
\usepackage[T1]{fontenc}
\usepackage[utf8]{inputenc}

\usepackage{graphicx}
\usepackage{color}
\usepackage{enumerate}
\usepackage{rotating}
\usepackage{tikz}
\usepackage{soul}
\usepackage{footnote}
\makesavenoteenv{tabular}
\makesavenoteenv{table}
\usepackage{amsmath, amsthm, amsfonts, amssymb}
\usepackage[normalem]{ulem}

\newtheorem{theorem}{Theorem}[section]

\newtheorem{corollary}[theorem]{Corollary}

\newtheorem{lemma}[theorem]{Lemma}
\newenvironment{proofcor}[2][Proof of Corollary]
{\par\noindent{\bf #1 #2.} }{\hspace*{\fill}\nolinebreak{$\Box$}\bigskip\par}

\begin{document}
\markboth{ }{}
\title{\bf Tight Bounds on the Complexity of Semi-Equitable Coloring of Cubic and Subcubic Graphs\footnote{Project has been 
partially supported by Narodowe Centrum Nauki under contract 
DEC-2011/02/A/ST6/00201}}
\date{}
\author{Hanna Furma\'nczyk\footnote{Institute of Informatics,\ University of Gda\'nsk,\ Wita Stwosza 57, \ 80-308 Gda\'nsk, \ Poland. \ e-mail: hanna@inf.ug.edu.pl},  \ Marek Kubale \footnote{Department of Algorithms and System Modelling,\ Gda\'nsk University of Technology,\ Narutowicza 11/12, \ 80-233 Gda\'nsk, \ Poland. \ e-mail: kubale@eti.pg.gda.pl}
}

\markboth{H. Furma\'nczyk, M. Kubale}{Tight Bounds on the Complexity of Semi-Equitable Coloring}

\maketitle

\begin{abstract}
A $k$-coloring of a graph $G=(V,E)$ is called semi-equitable if there exists a partition of its vertex set into independent subsets $V_1,\ldots,V_k$ in such a way that $|V_1| \notin \{\lceil |V|/k\rceil, \lfloor |V|/k \rfloor\}$ and $||V_i|-|V_j|| \leq 1$ for each $i,j=2,\ldots,k$. The color class $V_1$ is called non-equitable.
In this note we consider the complexity of semi-equitable $k$-coloring, $k\geq 4$, of the vertices of a cubic or subcubic graph 
$G$. In particular, we show that, given a $n$-vertex subcubic 
graph $G$ and constants $\epsilon > 0$, $k \geq 4$, it is NP-complete 
to obtain a semi-equitable $k$-coloring of $G$ whose non-equitable color class is of size $s$ 
if $s \geq n/3+\epsilon n$, and it is polynomially solvable if $s  \leq  n/3$.
\end{abstract}

\section{Introduction}
All graphs considered in this paper are finite, loopless, and without multiple edges. We refer the reader to \cite{west} for terminology in graph theory. 
We say that a graph $G = (V,E)$ is \emph{equitably $k$-colorable} if and only if its vertex set can be partitioned into independent sets $V_1,\ldots,V_k \subset V$ such that $|V_i| - |V_j| \in \{-1,0,1\}$ for all 
$i,j = 1,\ldots, k$. The smallest $k$ for which $G$ admits such a coloring is called the \emph{equitable chromatic number} of $G$ 
and denoted by $\chi_=(G)$. 
A graph $G$ on $n$ vertices has a \emph{semi-equitable coloring} if there exists a partition of its vertices into independent sets $V_1,\ldots,V_k \subset V$ such that one of these subsets, 
say $V_1$, is of 
size 
$s\notin \{\lfloor \frac{n}{k} \rfloor, \lceil \frac{n}{k} \rceil\}$, and the remaining subgraph $G - V_1$   is equitably $(k-1)$-colorable.
In what follows, such a color class $V_1$ will be called \emph{non-equitable}. 
These two models of graph coloring are motivated by applications in multiprocessor scheduling of unit-execution time jobs \cite{sch3,sch4}. 

In the following we will say that graph $G$ has a $(V_1, \ldots, V_k)$-coloring to express explicitly the partition of $V$ into $k$ independent 
sets. If, however, only cardinalities 
of color classes are important, we will use the notation of $[|V_1|, \ldots, |V_k|]$-coloring. For a given coloring, we call the difference
$\max\{|V_i|-|V_j|: i,j \in \{1,\ldots,k\}\}$ its \emph{color width}. Thus, a coloring of a graph is equitable if and only if the color width does not exceed 1. 

We mention the following two theorems on equitable graph coloring. First, Hajnal and  Szemeredi \cite{hfs:haj} proved

\begin{theorem}[\cite{hfs:haj}]
If $G$ is a graph with maximum degree $\Delta(G)$, $\Delta(G) \leq k$, then $G$ has an equitable $(k+1)$-coloring. \label{haj}
\end{theorem}

This theorem implies that every cubic graph, i.e. a regular graph of degree 3, has an equitable $k$-coloring for every $k \geq 4$. Kierstead et al. \cite{fast} gave a simple algorithm for obtaining 
such a coloring in $O(n^2)$ time. Secondly, Chen et al. \cite{clcub} proved 

\begin{theorem}[\cite{clcub}]
If $G$ is a connected $3$-chromatic cubic graph, then there exists an equitable $3$-coloring of $G$. \label{chen}
\end{theorem}

Actually, they proved \cite{clcub} that $\chi(G)=\chi_=(G)$ for any connected cubic graph $G$.
The proof starts from any proper 3-coloring of a connected cubic graph different from $K_4$ and $K_{3,3}$, and it relies on successive decreasing of the color width of this 
coloring by one or by two, step by step, until the coloring is equitable.
Moreover, Chen and Yen in \cite{yen} extended this result to disconnected subcubic graphs, where by a \emph{subcubic graph} we mean a graph $G=(V,E)$ with $\deg(v) \leq 3$ for all $v \in  V$.

\begin{theorem}[\cite{yen}]
A subcubic graph $G$ with $\chi(G)\leq 3$ is equitably 3-colorable if and only if exactly one of the following statements holds.
\begin{enumerate}
    \item No components or at least two components of $G$ are isomorphic to $K_{3,3}$.
    \item Only one component of $G$ is isomorphic to $K_{3,3}$ and $\alpha(G - K_{3,3}) > \frac{|V(G - K_{3,3})|}{3}>0$.
\end{enumerate}\label{yen}
\end{theorem}

By the above we immediately have the following corollary.
\begin{corollary}
If $G$ is a subcubic graph including neither $K_{3,3}$ nor $K_4$ as a component, then it admits an equitable $3$-coloring. \hfill $\Box$
\label{sub}
\end{corollary}

The problem of semi-equitable 3-coloring of connected cubic graphs was introduced in \cite{sch}. We have shown that every cubic graph with $t$ independent vertices has equitable
3-coloring for $t \in \{\lceil n/3 \rceil, \lfloor n/3 \rfloor \}$ and semi-equitable 3-coloring for $t \geq  2n/5$. In this note we extend those results to an arbitrary number $k \geq 4$ of colors 
and to, possibly disconnected, subcubic graphs. 
In contrast to equitable coloring not all cubic/subcubic graphs have a semi-equitable coloring (see $K_4$ for example). Therefore, in the 
following we assume that
all graphs under consideration have such a coloring. We will denote by $N(v)$ \emph{the $($open$)$ neighborhood} of the vertex $v\in V$, that is the 
set $\{u\in V: \{u,v\}\in E\}$. Let
$G_1 \cup G_2$ denote the \emph{union of graphs} $G_1=(V_1,E_1)$ and $G_2=(V_2,E_2)$ with disjoint vertex sets $V_1$ and $V_2$ and edge sets  $E_1$ and $E_2$, i.e. the graph $G$ with 
$V=V_1 \cup V_2$ and $E=E_1 \cup E_2$.
\begin{theorem}[\cite{yen}]
If two graphs $G_1$ and $G_2$ with disjoint vertex sets are both equitably $k$-colorable, then $G_1 \cup G_2$ is also equitably $k$-colorable. 
\end{theorem}

The rest of the paper is organized as follows. In Section \ref{np} we prove that, in contrast to equitable coloring, the problem of semi-equitable coloring 
becomes NP-complete for each $k \geq 3$. More precisely, we show that computing a semi-equitable $k$-coloring of a subcubic graph whose maximum color class is of size at 
least $(1/3 + \epsilon)n$ for any $\epsilon \in (0, 1/6)$ is NP-complete, if one exists. 
In Section \ref{semi4} we show how to obtain in $O(n^2)$ time a semi-equitable $k$-coloring of a subcubic graph with non-equitable color class of size at most $n/3$. Because we are interested in an algorithmic approach, in Appendix we reprove Corollary \ref{sub} by giving an appropriate algorithm resulting from a slight modification of Chen et al.'s proof \cite{clcub}. The computational complexity of the whole equalizing procedure is $O(n^2)$.


\section{NP-completeness of the problem} \label{np}

In this section we present one of the main results of the paper. We are interested in the computational complexity of deciding whether 
a subcubic graph $G$ has a semi-equitable $k$-coloring $(k \geq 4)$ with non-equitable color class of size $s$. In \cite{harder} we proved that the problem of deciding whether a cubic graph has a coloring of type $[4n/10,3n/10,3n/10]$ is NP-complete. 
In the following we strenghten and generalize this result to semi-equitable $k$-colorings $(k \geq 4)$ and subcubic graphs. We consider graphs not including $K_4$ as a component.
We say that $G$ is a $m$-\emph{divisible} graph if $|V(G)|$ is divisible by $m$. 

Let us define the following decision problem.

\vspace{0,5cm}

\emph{$(\frac{1}{3}+\epsilon)$-Semi-Equitable Coloring of a Cubic $m$-Divisible graph}
(SECCD)

Instance: A $m$-divisible cubic graph $G$, an integer $k \geq 4$, and $\epsilon >0$.

Question: Does $G$ have a semi-equitable $k$-coloring whose non-equitable color class is of size at least $(\frac{1}{3}+\epsilon)|V|$? 

\vspace{0.5cm}

We want to prove that SECCD is NP-complete. The following lemma states a strong relationship between our problem and the Stable Set Problem.

\begin{lemma}
Let $G$ be a cubic or subcubic graph, $k, s \in \mathbb{Z}^+$, and $k \geq 4$. Then $G$ has a semi-equitable $k$-coloring with non-equitable color class of size $s$ if and only if $G$ has an independent set of size $s$.\label{cors}
\end{lemma}
\begin{proof}
If $G$ has a semi-equitable $k$-coloring, $k \geq 4$, with non-equitable color class of size $s$ then there must exist an independent set of size $s$ in $G$.

Conversaly, if there is an independent set of size $s$ in $G$, it forms a non-equitable color class, say $V_1$.
If $k\geq 5$, the existence of an equitable $(k-1)$-coloring of the remaining subcubic graph $G - V_1$ follows from Theorem \ref{haj}. Let $k=4$. If $G-V_1$ fulfills the condition from Theorem \ref{yen}, then we have an equitable 3-coloring of $G-V_1$. Let us assume 
that $G-V_1$ is not equitably 3-colorable. 
This means that there is only one component of $G-V_1$ isomorphic to $K_{3,3}$ and $\alpha(G-V_1-K_{3,3}) \leq \frac{|V(G)|-s-6}{3}$ due to Theorem \ref{yen}. Then we try to exchange one of vertices from $V_1$ to another one in the copy of $K_{3,3}$ in $G-V_1$. If we succeed, there would be no component isomorphic to $K_{3,3}$ in $G-V_1$. Otherwise we conclude that exactly $s$ vertices from $V_1$ belong to $s$ subgraphs isomorphic to $K_{3,3}$ in $G$ and the subgraph $G-V_1$ can be expressed as $sK_{2,3} \cup K_{3,3} \cup H$,  where $H$ is a subgraph (possobly empty) of $G-V_1$ that is free from $K_{3,3}$. If $V(H)$ includes at least one vertex, we exchange one vertex from $V_1$ with any vertex of $H$. After this exchange a new graph $G-V_1$ includes $2K_{3,3}$ as a subgraph and such a graph is equitably 3-colorable due to Theorem \ref{yen}. If $|V(H)|=0$, then we have $G-V_1 = sK_{2,3} \cup K_{3,3}$ and the independence number of $G-V_1-K_{3,3}$ is equal to $3s$ while $|V(G-V_1-K_{3,3}|=5s$. This means that the condition from Theorem \ref{yen} is fulfilled, as $3s > 5s/3$, and the graph $G-V_1$ is equitably 3-colorable.
\end{proof}

Now, we consider the following subproblem of the well known Maximum Independent Set (MIS) problem restricted to cubic graphs.

\vspace{0,5cm}

\emph{Maximum Independent Set in a Cubic $m$-Divisible graph} (MISCD)

Instance: A $m$-divisible cubic graph $G$, $l \in \mathbb{Z}^+$.

Question: Does $G$ have an independent set of size at least $l$? 

\begin{lemma}
Problem \emph{MISCD} is \emph{NP}-complete. \label{lm32}
\end{lemma}
\begin{proof}
Note that the MIS problem is NP-complete \cite{garey} even if $G$ is 3-regular. We show that it remains so if $m | n$ (MISCD). Let us denote by $Cub_{t,t}$ any
cubic bipartite graph on $2t$ vertices such that $n+2t$ is divisible by $m$, $3 \leq t \leq m/2+2$. We remark that a bipartite cubic graph has a perfect matching, hence by K\"{o}nig's theorem \cite{west} the size of a maximum independent set in such a graph on $2t$ vertices is $t$. The statement holds because if we consider the graph $G \cup Cub_{t,t}$, the number of 
vertices in the new graph is divisible by $m$ and, moreover, $G$ has an independent set of size at 
least $l$ if and only if $G \cup Cub_{t,t}$ has an independent set of size at least $l+t$. 
\end{proof}

We will see that MISCD remains NP-complete on the subset of instances where 
$l=(\frac{1}{3}+\epsilon)|V|$, $\epsilon >0$. Let us formally define the new decision problem.
\vspace{0,5cm}

\emph{$(\frac{1}{3}+\epsilon)$-Independent Set in a Cubic 6-Divisible graph} (ISCD)

Instance: A 6-divisible cubic graph $G$, $\epsilon >0$.

Question: Does $G$ have an independent set of size at least $(\frac{1}{3}+\epsilon)|V|$? 

\begin{lemma}
Problem \emph{ISCD} is \emph{NP}-complete.
\label{npcom}
\end{lemma}
\begin{proof}
Let $G$, $l$, and $\epsilon$ be an instance of the MISCD problem with $m=6$. Let $n_G=|V(G)|$ and let $p = n_G \lceil 1/\epsilon \rceil$. Obviously, $1/p< \epsilon$. We will reduce the question about existence in $G$ an independent set of size at least $l$ to the question whether the subsequently defined cubic graph $H$ has an independent set of size at least $(1/3+1/p)n_H$, where $n_H = |V(H)|$. Let $q = |(1/3+1/p)n_G - l|$.

If $l \geq (1/3+1/p)n_G$, $H$ is the union of $G$ and $pq/6$ copies of the $P$ graph, where $P$ is a 6-vertex prism (2 triangles joined by a 3-matching). This results in increasing the order of the graph by $pq$, i.e. $n_H = n_G+pq$, which is divisible by $n_G$. Each of these copies of $P$ provides exactly 2 new vertices in any maximum independent set. Thus, if $G$ has an independent set on at least $l$ vertices than there exists in $H$ an independent set of size at least
$$l+pq/3 = (l-q)+(q+pq/3) = (1/3+1/p)n_G+(q+pq/3) =$$ $$=(1/3+1/p)n_G+pq(1/3+1/p) = (1/3+1/p)n_H,$$
and vice versa.

If $l < (1/3+1/p)n_G$, $H$ is the union of $q$ copies of $p/6-2$ prisms and two graphs $K_{3,3}$. This time each such subgraph $(p/6-2)P\cup 2K_{3,3}$ on $p$ vertices provides $(p/6-2)2+6 = p/3+2$ new independent vertices. Thus, if $G$ has an independent set of size at least $l$ then there exists in $H$ an independent set of size 
$$l+q(p/3+2) = (l+q)+q(p/3+1) = (1/3+1/p)n_G+ (q +pq/3) = (1/3+1/p)n_H,$$
and vice versa.            
\end{proof}

As a direct consequence of Lemmas \ref{cors} and \ref{npcom} we can conclude
\begin{theorem}
For any fixed $\epsilon \in (0,1/6)$ problem \emph{SECCD} is \emph{NP}-complete.     \label{twnpsemi} \hfill $\Box$
\end{theorem}

\section{Semi-equitable $k$-coloring of subcubic graphs}\label{semi4}
In Section \ref{np} we have proved that for any constants $k \geq 4$, $\epsilon \in (0,1/6)$ and any subcubic graph $G$, the problem of deciding if $G$ has a $(1/3+\epsilon)$-semi-equitable $k$-coloring is 
NP-complete. It turns out that if we diminish the size of non-equitable color class slightly, namely to $\lceil n/3 \rceil$, then the problem becomes polynomially solvable for any
$k \geq 4$.

\begin{theorem}
Given a $n$-vertex cubic or subcubic graph $G$ not including $K_4$ neither $K_{3,3}$, a constant $k \geq 4$, and an integer $s \leq \lceil n/3 \rceil$, finding a semi-equitable $k$-coloring of $G$ with non-equitable color class of size $s$
is solvable in $O(n^2)$ time.
\end{theorem} 
\begin{proof}
First, we have to determine an independent vertex set $I$ of size $s$, $s \leq \lceil n/3 \rceil$. This is an 
easy task, if we apply the following greedy approach. First, we find a vertex $v$ of minimum degree and delete it from the graph together with its
neighborhood $N(v)$. Note that, in all steps except the first, at most 3 vertices are deleted. Then, we repeat this step until $s$ independent vertices are found. If it is not the case and one more independent vertex is needed, we can apply an equalizing 3-coloring procedure given in the proof of Corollary \ref{sub} (see Appendix).
It is easy to see that $s$ must satisfy $s \leq \lceil n/3 \rceil$ and this bound is tight. Next, 
we color graph 
$G - I$ equitably with $k-1$ colors. If $k \geq 5$ then $(k-1)$-coloring of $G - I$ is 
guaranteed by Theorem \ref{haj}. Such a coloring can be obtained in $O(n^2)$ time. 
If $k=4$, $G-I$ can be properly colored with 3 colors, as a subcubic graph 
different from $K_4$. Since $G - I$ is also different from $K_{3,3}$, we can apply the procedure from the proof of Corollary \ref{sub} for equalizing a given 3-coloring of $G - I$. Since this procedure is executed at most twice, the complexity of $O(n^2)$ follows. 
%
\end{proof}

One may ask about a semi-equitable 3-coloring of a subcubic graph $G$. The problem was discussed in \cite{sch}, where we proved 
\begin{theorem}[\cite{sch}]
If an $n$-vertex cubic graph $G$ has an independent set $I$ of size $|I| \geq 2n/5$, then it has a semi-equitable coloring of type 
$[|I|,\lceil \frac{n-|I|}{2} \rceil,\lfloor \frac{n-|I|}{2} \rfloor]$. \hfill $\Box$
\end{theorem}

Moreover, we noticed that a cubic graph has an independent set of size $|I| \geq 2n/5$ almost surely. This is so because Frieze and Suen \cite{freize} proved that for random cubic graphs $G$ their independence
number $\alpha(G)$ fulfills the inequality $\alpha(G) \geq­ 4.32n/10-\epsilon n$ for any $\epsilon > 0$ almost surely. Thus a random cubic graph is very likely to have an independent set
of size $s \geq­ 2n/5$ and the probability of this fact increases with $n$.

Tables \ref{tab_cub} and \ref{tab_sub} gather the computational complexity status for semi-equitable $k$-colorings with non-equitable class of size 
$s$ of $n$-vertex cubic and subcubic graphs, respectively. 

\begin{table}[h]
\begin{center}
\begin{tabular}{|c|*{4}{c|}}\hline
$k$ & $s \leq \lceil \frac{n}{3} \rceil$ & $\lceil \frac{n}{3} \rceil < s < \frac{n}{2}$ & $s=\frac{n}{2}$ & $\frac{n}{2} <s$\\ \hline
3 & $O(n)$\footnote{the linear solution 
concerns bipartite cubic graphs only}/-\footnote{the corresponding solution may not exist} & NPC & $O(n)^1$/-$^2$ & -$^2$\\ \hline
$\geq 4$ & $O(n^2)$ & NPC & $O(n)^1$/-$^2$ &-$^2$\\ \hline
\end{tabular}
\caption{The computational complexity of semi-equitable $k$-coloring of cubic graphs.}\label{tab_cub}
\end{center}
\end{table}

\begin{table}[htb]
\begin{center}
\begin{tabular}{|c|*{3}{c|}}\hline
$k$ & $s \leq \lceil\frac{n}{3}\rceil$ & $\lceil\frac{n}{3}\rceil < s \leq \lfloor
\frac{3n}{4}\rfloor$ & $\lfloor\frac{3n}{4}\rfloor <s$\\ \hline
3 & -$^2$ &  NPC & -\footnote{the corresponding solution does not exist in the case of connected graphs}\\ \hline
$\geq 4$ &  $O(n^2)$ & NPC & -$^3$\\ \hline
\end{tabular}
\caption{The computational complexity of semi-equitable $k$-coloring of subcubic graphs.}\label{tab_sub}
\end{center}
\end{table}
\textbf{Acknowledgments.} The authors deeply thank the anonymous referee for her/his careful reading of the manuscript and her/his many insightful comments and suggestions.
Moreover, the authors thank Professor Adrian Kosowski for taking great care in reading our manuscript and making several useful
suggestions improving the presentation.

\section*{Appendix - the proof of Corollary \ref{sub}}
\setcounter{section}{4}
First, we recall some notations used by Chen et al. \cite{clcub}. Given $A$ and $B$ disjoint subsets of $V(G)$,
let $Ak_GB$ denote the set $\{x \in A: x \text{ is adjacent to}$ exactly  $k \text{ vertices of } B \text{ in } G\}$, while $A \deg_G B$ denotes the set of all vertices in $A$ having all its 
neighbors in set $B$, namely $A \deg_G B=$ $\{x \in A: N(x) \subset B\}$. When it is clear we will use $AkB$ and $A \deg B$.
Given a coloring of $G$, the notation $A \Leftrightarrow B$ means that we exchange the color of vertices in $A$ into the color of vertices in $B$ and vice versa. 
A one-way arrow $A \Leftarrow B$ means that we change the color of $B$ into the color of $A$. We write $A \Leftarrow x$ when $B=\{x\}$.
By $G(X,Y)$ we will denote a bipartite graph whose vertices can be divided into two disjoint sets $X$ and $Y$ such that every edge
connects a vertex in $X$ to one in $Y$. 

\begin{lemma}[\cite{clcub}]
Let $G(X,Y)$ be a connected bipartite subcubic graph such that $|X|\geq |Y|$. Then 
$|X|-|Y| \leq |Y3X|+1$. \hfill $\Box$ \label{lm}
\end{lemma}

This lemma can be extended to the following one.

\begin{lemma}
Let $G(X,Y,Z)$ be a tripartite subcubic graph and let $H(X',Y')$, $|X'|\geq |Y'|$, be its connected bipartite subgraph such that $X' \subset X$ and $Y' \subset Y$, and let $t=|Y \deg_H X \cap Y'|$. Then $|X'|-|Y'| \leq t+1$. \label{lmex}
\end{lemma}
\begin{proof}
Let $e$ be the number of edges in $H$. Since $H$ is connected, we have $e \geq |X'|+|Y'|-1$. On the other hand, 
the number of edges can 
be bounded from above by $e \leq 3t+2(|Y'|-t)=2|Y'|+t $. We get the thesis by combining these two inequalities.
\end{proof}

\begin{proofcor}{\ref{sub}}
We start with $(A,B,C)$ a proper 3-coloring of $G$ with $|A| \geq |B| \geq |C|$. If the coloring is not equitable we will decrease the color width of the coloring by 1 or 2.
We repeat the color width decreasing procedure until the obtained coloring is equitable.

Let us assume that the 3-coloring is not equitable, this means that $|A|-|C| \geq 2$. We may assume that there is no isolated vertex in $G$, since if we have an equitably 
$k$-colored graph $G$ and we 
add one isolated vertex, then we can always color the isolated vertex in such a way that the whole graph is equitably $k$-colored.
Thus, we have $1 \leq \deg(x) \leq 3$ for 
each $x \in V(G)$. Now, we consider the following steps (cases).
\begin{enumerate}
\item If $A \deg B \neq \emptyset$, then: Choose $x \in A \deg B$ and do $C \Leftarrow x$. 
The color width of the new coloring $(A - \{x\}, B, C \cup \{x\})$ decreases in one or two units. \\Henceforth, in the following we assume $A \deg B = \emptyset$.

\item If $C3A = \emptyset$ then: Consider the bipartite subgraph $G(A,C)$ induced by $A \cup C$. Find a connected component $G(A',C')$ 
of $G(A,C)$ such that $|A'| =|C'|+1$ and do $A' \Leftrightarrow C'$.

Let us prove that there exists such a connected component $G(A',C')$ of $G(A,C)$.  In fact, since $C3A = \emptyset$, there is a component $G(A',C')$ such that 
$|A'| =|C'|+1$, by Lemma \ref{lm}. Then the color width of the new coloring decreases in one or two units. \\Henceforth, in the following we are assuming that $C 3 A \neq \emptyset$.

\item If $|C \deg B| \geq |C 3 A|$ then: Consider the bipartite subgraph $G(A,C)$ induced by $A \cup C$. Find a connected component $G(A',C')$ 
of $G(A,C)$ such that $|A'| \geq |C'|+1$, choose $S \subset (C \deg B) - C'$ such that $|S|=|A'|-|C'| - 1$ and do $A' \Leftrightarrow C' \cup S$.

Let us prove that there exists such a connected component $G(A',C')$ of $G(A,C)$ and such a subset $S$. In fact, since 
$|A| \geq |C| + 2$, there must exist a connected 
component $G(A',C')$ of $G(A,C)$ for which $|A'|\geq |C'|+1$. 
From the assumption of this case we have $|C' 3 A'| \leq |C \deg B|$. Moreover, we have
$|A'|-|C'|=t'$, where $t'\leq |C \deg B|+1$, by Lemma \ref{lmex}, and the subset $S \subset C \deg B - C'$ such that $|S|=t'-1$ may be chosen. 
Then the color width of the new coloring decreases by one unit. 
\\Henceforth, in the following we assume $|C \deg B| \leq |C 3 A| - 1$.

\item If $B \deg A = \emptyset$ then: If $|B| = |C|$ do $B \Leftrightarrow C$. Otherwise consider the bipartite subgraph $G(B,C)$ induced by $B \cup C$. Find a 
connected component $G(B',C')$ of $G(B,C)$ such that $|B'|\geq |C'|+1$, choose $S \subset (C \deg A) - C'$ such that $|S|=|B'|-|C'|$ and do 
$B' \Leftrightarrow C' \cup S$.

Let us prove that there exists such a connected component $G(B',C')$ of $G(B,C)$ and such a subset $S$. 
Again, there must exist a connected component $G(B',C')$ of $G(B,C)$ for which $|B'|\geq |C'|+1$. Since 
$|C' \deg B'| \leq |C\deg B|$, 
we have $|B'|-|C'|=t$, where $t \leq |C\deg B|+1$. Therefore, a subset $S \subseteq C \deg A - C'$ such that $|S|=t$ may be chosen. 

Note that the new coloring does not decrease the color width but verifies $B \deg A \neq \emptyset$. \\Henceforth, in the following we assume $B \deg A \neq \emptyset$.

\item If $A \deg C \neq \emptyset$ then: If $|A|=|B|$ do $C \Leftarrow x$ for any $x \in B \deg A$, otherwise choose $x \in B \deg A$ and $y \in A \deg C$ and do $B 
\Leftarrow y$ and $C \Leftarrow x$.

Of course, all these operations are possible and the color width of the new coloring decreases. \\Henceforth, in the following we assume $A \deg C = \emptyset$.
\end{enumerate}

\noindent From now on, we apply the following steps of the procedure on a one chosen  connected component $G(A',B')$ of $G(A,B)$ such that $|A'|\geq |B'|+1$.

\begin{enumerate}
\setcounter{enumi}{5}
\item If $B' \deg A' = \emptyset$ then: Do $C \Leftarrow x$ and $A' \Leftrightarrow B'$ for any $x \in B \deg A$.

Let us prove the correctness of this step. Since $B' \deg A' = \emptyset$ then we have $|A'| = |B'|+1$, by Lemma \ref{lmex}. 
Thus, the color width of the new coloring decreases in one or two units. An example of such a situation is given in Figure \ref{rys6}. Edges of a subgraph $G(A',B')$ are depicted by dotted line.
\\Henceforth, in the following we may assume that $B' \deg A' \neq \emptyset$.

\begin{figure}[h]
\begin{center}
\includegraphics[scale=0.9]{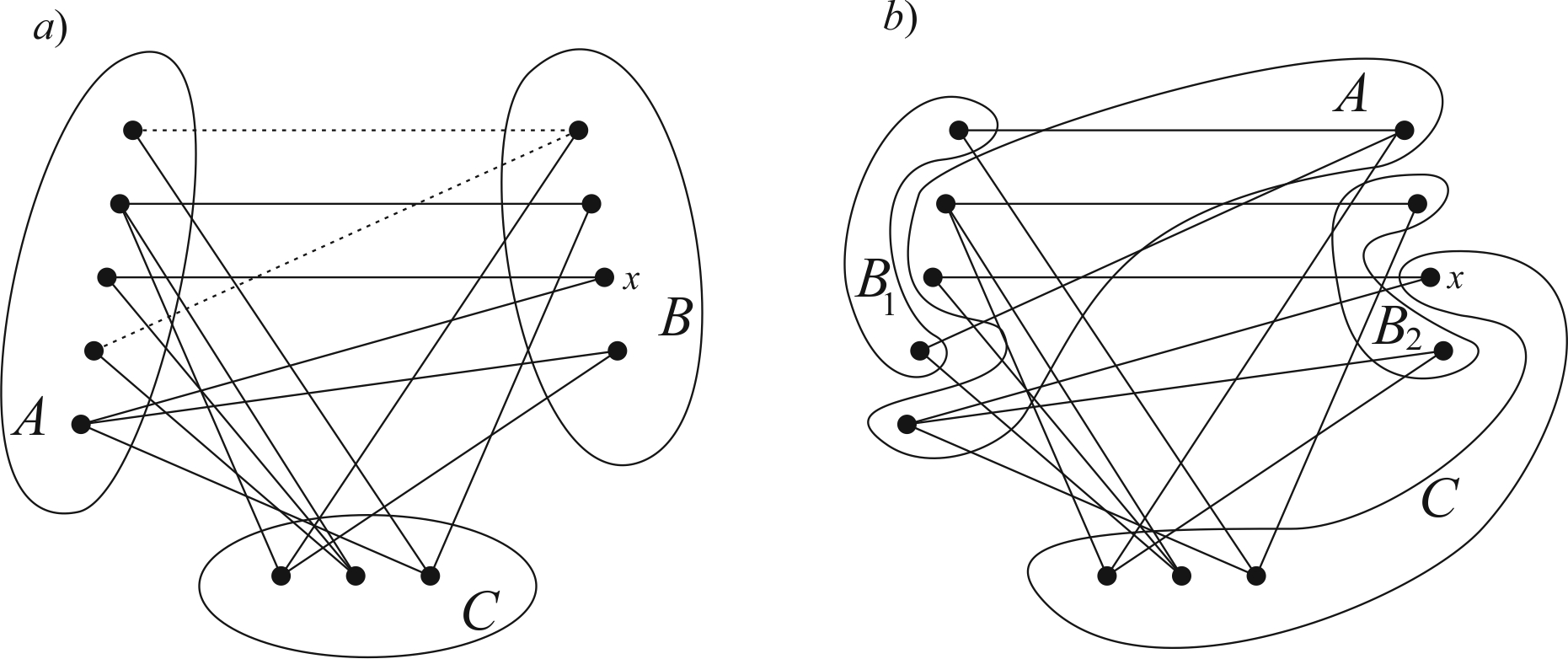} 
\caption{An example for the proof of Corollary \ref{sub} Case 6. A graph with its (a) $[5,4,3]$-coloring and (b) $[4,4,4]$-coloring; $B=B_1 \cup B_2$.}\label{rys6}
\end{center}
\end{figure}

\item If there exists some $x \in B' \deg A'$ such that one of its neighbors, say $y$, satisfies $y \in A'1B'$, then: Do 
$B \Leftarrow y$ and $C \Leftarrow x$.

Of course, such operations are possible and the color width of the new coloring decreases in one or two units. \\Thus, since $A \deg B = \emptyset$ we may assume that 
for every $y \in N(x)$ we have $y \in A'2B'$.

\item If $|B' \deg A'|\geq 2$ and for every $x_1,x_2 \in B' \deg A'$ it holds $N(x_1) \cap N(x_2) \neq \emptyset$, then: Choose any $w \in C \deg A$ and $u \in N(x_1) \cap N(x_2)$.

First, we will show that we may assume $N(x_1)\neq N(x_2)$. 
Indeed, if $\deg(x_1)=\deg(x_2)=1$ then we have $N(x_1)= N(x_2)$, and therefore $|A'|=|N(x_1)|=1$ and $B'=\{x_1,x_2\}$, which contradicts
$|A'| > |B'|$.  Similarly, if $\deg(x_1)=\deg(x_2)=2$ then we also have a contradiction with $|A'| > |B'|$. So let $\deg(x_1)=\deg(x_2)=3$ 
and $N(x_1)= N(x_2)$. 
First, we do an exchange $B \Leftarrow z$ and $C \Leftarrow x_2$ for any $z\in C3A$. Then we have $N(z) \neq N(x_1)=N(x_2)$, since otherwise $N(x_1)=N(x_2)=N(z)$ would force $G=K_{3,3}$, which is excluded by the assumption of the theorem.
Therefore, we have $N(x_1)\neq N(x_2)$, as claimed. 

\begin{itemize}
\item If $u \not \in N(w)$ then: Do $B \Leftarrow u$, $B \Leftarrow w$ and $C \Leftarrow \{x_1,x_2\}$. 

The color width is decreased because the cardinality of color class $A$ was decreased by 1, the number of vertices colored with 2 remains unchanged, while the cardinality of color class $C$ was increased by 1.

\item If $u \in N(w)$ and $ B' \deg A' \backslash \{x_1,x_2\} \neq \emptyset$, then:  Choose $x_3 \in B' \deg A'\backslash \{x_1,x_2\}$ and do $B \Leftarrow w$, $C \Leftarrow u$, and $C \Leftarrow x_3$.

This step is correct because $u$ is not adjacent to $x_3$. Moreover, the number of vertices with 1 is decreased by 1 (vertex $u \in A$ was recolored) while the number of vertices in class C was increased by 1.
\end{itemize}

\item Then, we can assume that for every $\{x_1,x_2\} \subset B' \deg A'$ we have that $N(x_1) \cap N(x_2) \neq \emptyset$ or $B' \deg A' = \{x_1,x_2\}$ and for any $w \in C \deg A$ and any $ u \in N(x_1) \cap N(x_2)$ we have $u \in N(w)$. Note, that it concerns also the case when $|B' \deg A'|=1$, i.e. $x_1=x_2$. Then:

\begin{itemize}
    \item delete $B' \deg A'$ and all those vertices which are adjacent to the two vertices of $B' \deg A'$ from $G(A',B')$. 
    
    Since we have assumed earlier that $x \in B' \deg A'$ and $y \in N(x)$ imply $y \in A'2B'$ for any $x$ and $y$, so each of the remaining vertices will have degree 1 or 2. Hence the resulting graph $G'(A',B')$ decomposes into maximal paths. 
    
    \item There must exists a path with initial and terminal vertices in the same partitions, otherwise $|B'|=|A'|$, which is a contradiction. Without loss of generality, let $u_0,v_1,u_2,v_3, \ldots, u_{2m}$ be a maximal path of $G'(A',B')$ such that $u_0$ is adjacent to $x_1 \in B' \deg A'$ and $u_{2m} \in A'$. 
    
    \begin{itemize}
        \item If $u_{2m} \notin N(x_2)$ for some $x_2 \in B' \deg A'$ and $x_2 \neq x_1$, then: Do $C \Leftarrow x_1$ and $A'' \Leftrightarrow B''$, where $A''=\{u_0,u_2,\ldots,u_{2m}\} \subseteq A'$ and $B''=\{v_1,v_3,\ldots,v_{2m-1} \}\subseteq B'$. 
        
        \item If $u_0 \in N(x_1) - N(x_2)$ and $u_{2m} \in N(x_2)-N(x_1)$ for distinct $x_1,x_2 \in B' \deg A'$, then:
        \begin{itemize}
            \item If $w \in N(v)$ for some $v \in A''$, then: Do $B \Leftarrow w$, $C \Leftarrow \{v,x_2\}$. 
            
            In this case $v$ is not adjacent to at least one of $x_1$ and $x_2$, say $x_2$ and the step is correct.
            
            \item If $w \in N(v)$ for each $v \in A''$, then: Do $B \Leftarrow w$, $C \Leftarrow \{x_1,x_2\}$, $A '' \Leftrightarrow B''$. 
            
            $w$ is independent in $A''$, thus the exchange is correct.
 \end{itemize}
 \end{itemize}
\end{itemize}
\end{enumerate}
An example of such a situation from Case 9 is given in Figure \ref{ex2}. Here $G(A',B')$ consists of $A'=\{u, a_1,\ldots, a_5\}$, $B'=\{x_1,x_2,b_1, \ldots, b_3\}$ and dashed and dotted edges. Graph $G'(A',B')$ consists of vertices $A''=\{a_1,\ldots,a_5\}$, 
$B''=\{b_1, \ldots, b_3\}$, and edges drawn with a dashed line. We choose as the maximal path mentioned above - path $(a_1,b_2,a_4,b_1,a_3)$. We rename these vertices as 
$(u_0,v_1,u_2,v_3,u_4)$. Here $u_0 \in N(x_1)\backslash N(x_2)$ and $u_{2m}=u_4 \in N(x_2) \backslash N(x_1)$.
Since $w \in N(u_0)$, we do: $B \Leftarrow w$, $C \Leftarrow \{u_0,x_2\}$ - the final result is given in Fig.~\ref{ex2}(b).
\begin{figure}[h]
\begin{center}
\includegraphics[scale=0.9]{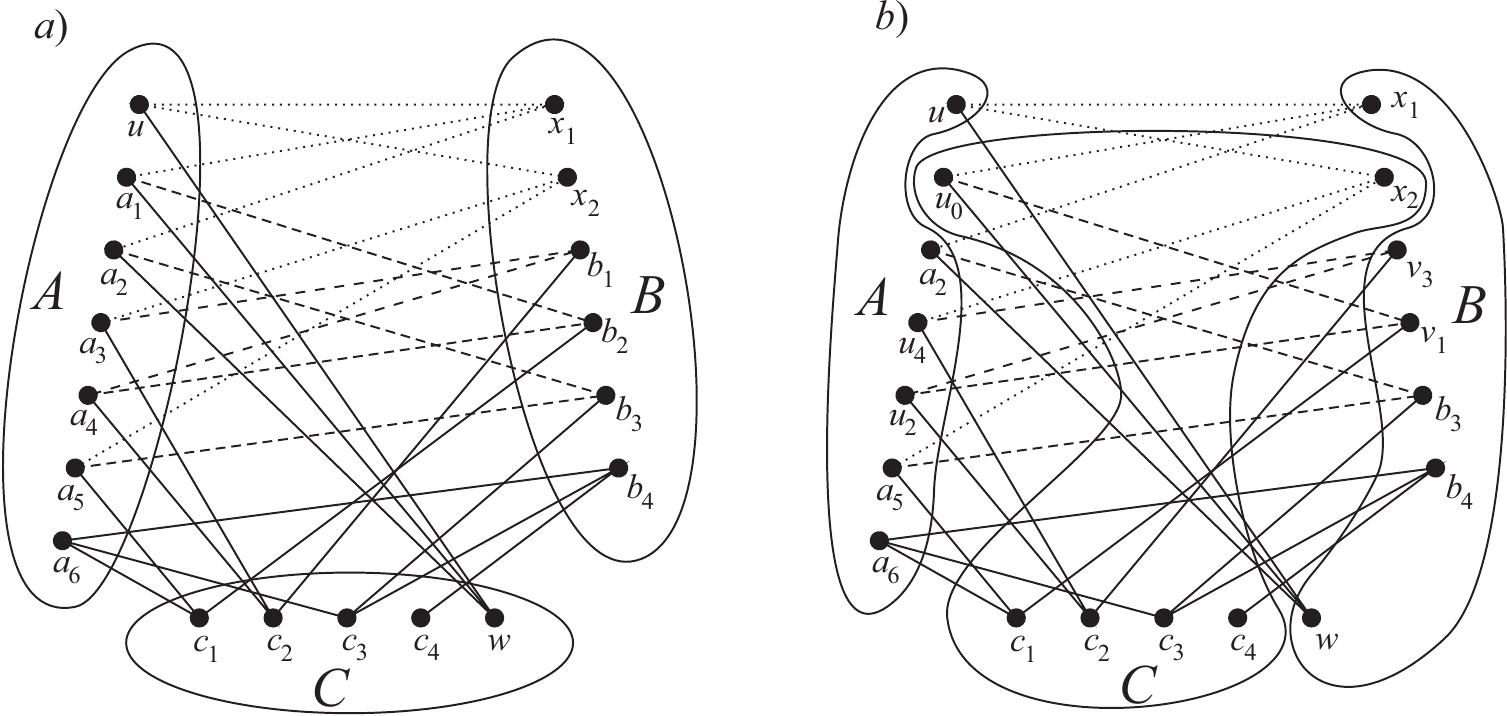} 
\caption{An example of the last step in the proof of Corollary \ref{sub}.}\label{ex2}
\end{center}
\end{figure}

We finally conclude that the color width has been decreased in all the cases of non-equitable 3-coloring of subcubic graph and therefore the proof is complete.
\end{proofcor}

It is easy to see that a single step of decreasing the color width can be done in linear time. Since such a decreasing procedure must
be applied at most $n/3-1$ times, we have $O(n^2)$ as the computational complexity of the whole equalizing procedure.

Of course, if a graph is disconnected, each connected component may be colored and equalized separately. Finally, we may merge the equitable colorings of components into the equitable 3-coloring of the whole graph.

\end{document}